%% file: lundervold2012bea.tex
\documentclass{article}
\pdfoutput=1

\usepackage{a4wide, amsthm, color, verbatim, multicol, hanshopf, float, subfig, shuffle}
\usepackage{difftrees}

\usepackage[small]{diagrams}

\newtheorem{theorem}{\bf{Theorem}}[section]
\newtheorem{lemma}[theorem]{\bf{Lemma}}
\newtheorem{proposition}[theorem]{\bf{Proposition}}

\theoremstyle{definition}
\newtheorem{example}[theorem]{\bf{Example}}

\newtheoremstyle{remark}{\topsep}{\topsep}
{}
{} 
{\bfseries} 
{.} 
{ } 
{\thmname{#1}\thmnumber{ #2}\thmnote{ #3}}

\theoremstyle{definition}
\newtheorem{remark}[theorem]{\bf{Remark}}

\theoremstyle{definition}
\newtheorem{definition}[theorem]{\bf{Definition}}

\def \B+{\operatorname{B}}
\def \Hom{\operatorname{Hom}}
\def \End{\operatorname{End}}

\def \Id{\operatorname{Id}}
\def \k{\operatorname{k}}

\def \NT{\operatorname{NT}}
\def \Lie{\operatorname{Lie}}
\newcommand{\car}{\curvearrowright}

\begin{document}

\thispagestyle{empty}

\title{Backward error analysis and the substitution law for Lie group integrators}

\date{}



\author{Alexander Lundervold \footnote{Corresponding author. Department of Mathematical Sciences, Norwegian University of Science and Technology, N-7491 Trondheim, Norway. alexander.lundervold@gmail.com} \and Hans Munthe-Kaas \footnote{Department of Mathematics, University of Bergen, N-5020 Bergen, Norway. hans.munthe-kaas@math.uib.no}}
\maketitle

\begin{abstract} Butcher series are combinatorial devices used in the study of numerical methods for differential equations evolving on vector spaces. More precisely, they are formal series developments of differential operators indexed over rooted trees, and can be used to represent a large class of numerical methods. The theory of backward error analysis for differential equations has a particularly nice description when applied to methods represented by Butcher series. For the study of differential equations evolving on more general manifolds, a generalization of Butcher series has been introduced, called Lie--Butcher series. This paper presents the theory of backward error analysis for methods based on Lie--Butcher series. 
\end{abstract}

\section{Introduction}
A fundamental tool in the field of numerical integration of ordinary differential equations on $\RR^n$ is the theory of \emph{Butcher series} (B-series). These are formal series expansions of vector fields and flows, expanded over the set of rooted trees. Many numerical methods can be formulated in terms of B-series, and they can be used to, for example, study order theory, structure preserving properties of integrators, backward error analysis and modified vector fields \cite{butcher1972aat, hairer2006gni, chartier2007nib, chartier2005asl, chartier2009aat, chartier2010aso, lundervold2011oas}. In the more general setting of differential equations of the form
\begin{equation}\label{diffeq}
y' = F(y), \hspace{0.2cm} y \in M, \,\, F: M \rightarrow TM,
\end{equation}
where $\M$ is a (homogeneous) manifold and $F$ a vector field on $\M$, the role of B-series is played by the \emph{Lie--Butcher series} (LB-series) \cite{munthe-kaas2008oth, lundervold2009hao, lundervold2011oas}. Considering the importance of classical B-series, LB-series are objects of great interest. 

The B-series are based on the \emph{elementary differentials} associated to vector fields, and these can be constructed as homomorphisms from the free \emph{pre-Lie algebra} (or \emph{Vinberg algebra}) into the pre-Lie algebra of vector fields \cite{calaque2009tih}. In the setting of LB-series we get a similar picture, only now the pre-Lie algebras are replaced by the so-called \emph{post-Lie algebras}, defined in \cite{vallette2007hog, munthe-kaas2012opl}.

In the present paper we will explore the \emph{substitution law} for Lie--Butcher series, formulated in the language of enveloping algebras of post-Lie algebras: the \emph{D-algebras} of \cite{munthe-kaas2008oth}. Once the substitution law is understood, it can be applied to \emph{backward error analysis}. The basic idea of backward error analysis is to interpret the numerical solution of a differential equation as the exact solution of a modified equation, and then use this equation to study the numerical method. Analogous to classical backward error analysis (as developed in \cite{hairer1994bao, hairer2006gni, chartier2007nib, calaque2009tih}), its generalization to the Lie group setting has a particularly nice description for methods based on Lie--Butcher series.

Note that the construction of series expansions in the present paper is purely formal: there will be no study of convergence. This separation between the algebraic and the analytic framework for backward error analysis is also present in the setting of B-series, where the main algebraic references are \cite{hairer2006gni, chartier2007nib, calaque2009tih, chartier2010aso} and the analytic references are \cite{benettin1994oth, reich1999bea, hairer2006gni}. An analytic study of backward error analysis for Lie group methods can be found in \cite{faltinsen2000bea}.

The present study of  the backward error and substitution law for Lie group integrators is interesting from a purely algebraic point of view, as this work provides an explicit description of automorphisms of post-Lie algebras. From a numerical point of view, the theory has several applications. The algebraic structures of backward error analysis is important in the analysis of numerical integration algorithms. Additionally, in the case of classical B-series, such algebraic techniques have recently been applied more directly as a computational tool~\cite{chartier2007nib}. Similar techniques in the setting of Lie group integrators is a promising approach to structure preserving integration of problems of computational mechanics, such as Lie-Poisson systems.

\section{Lie--Butcher series}

In this section we will define D-algebras, and show how they give rise to Lie--Butcher series. In the next section we will apply them to the study of the substitution law and backward error analysis for Lie group integrators on manifolds.

\subsection{Trees and D-algebras}

\paragraph{Ordered rooted trees and forests.}\label{sect:trees}
Some basic definitions follow. For a more comprehensive introduction to the combinatorics of trees applied to numerical integration, see \cite{butcher2008nmf} or \cite{iserles2000lgm}. Let $\OT$ denote the alphabet of all ordered (i.e. planar) rooted trees:
\[\OT = \{\ab, \, \aabb, \, \aababb, \, \aaabbb, \, \aabababb, \, \aabaabbb, \, \aaabbabb, \, \aaababbb, \, \ldots \}.\]
The root is the bottom vertex and we consider the trees to grow upwards from the root. The trees being ordered implies that $\aabaabbb\neq\aaabbabb$. This is different from classical B-series theory, where the order of the branches is of no significance. Let $\OF$ denote the set of ordered forests, i.e. all possible empty and non-empty words written with letters from the alphabet $\OT$:
\[\OF = \left\{\one, \,\, \ab, \,\, \ab\,\ab, \,\, \aabb, \,\, \ab\,\ab\,\ab, \,\, \aabb\,\ab, \,\, \ab\,\aabb, \,\, \aababb, \,\, \aaabbb, \,\, \cdots\right\},\] 
where $\one$ denotes the empty word. On $\OF$ we define the \emph{concatenation product} $\omega_1,\omega_2 \mapsto \omega_1\omega_2$, which creates a longer word by joining $\omega_1$ and $\omega_2$ end-to-end. This is an associative, non-commutative product with unit $\one$. Let $B^+ \colon \OF\rightarrow \OT$ denote the operation of adding a root to a word, e.g. $B^+(\ab\aabb)=\aabaabbb$. All of $\OF$ is generated from $\one$ by concatenation and adding roots. The \emph{order} of a forest, $|\omega| = |\tau_1\ldots \tau_k|$, is defined by the recursion $|\one|=0$, $|\tau_1\ldots\tau_k| = |\tau_1|+\cdots+|\tau_k|$, $|B^+\omega| = |\omega|+1$, i.e.\ the order counts the number of vertices in a forest. Let $\k$ be a field of characteristic 0, e.g.\ $\k=\RR$ or $\k=\CC$. The $\k$-vector space of all finite $\k$-linear combinations of elements in $\OF$ is the non-commutative polynomial ring over $\OT$ \footnote{$\N$ with concatenation product can equivalently be defined as the linear space spanned by trees, $V=k\{\OT\}$, equipped with a tensor product. Hence $\N$ can be defined as the tensor algebra on $V$. However, because we need other tensor products later we prefer the definition via concatenation of words.}, denoted by $\N = \k\langle \OT \rangle$. The $\k$-vector space of infinite linear combinations of $\OF$ is $\N^* = \k\langle\langle\OT\rangle\rangle$. $\N^*$ is  the dual space of $\N$, with the dual pairing $\langle\cdot,\cdot\rangle \colon \N^*\times \N\rightarrow \k$ defined such that the words in $\OF$ form a orthonormal basis: $\langle\omega_1,\omega_2\rangle=0$ if $\omega_1\neq\omega_2$, and $\langle\omega,\omega\rangle=1$. Thus for $a\in \N^*$ we have $a(\omega) = \langle a,\omega\rangle$ and $a = \sum_{\omega\in \OF} a(\omega)\omega$. In the latter sum we understand $\N^*$ as the projective limit $\N^* =\underleftarrow{\lim} \, \N_k$, where $\N_k = \mbox{span}\{\omega\in \OF \colon |\omega|\leq k\}$. An infinite $a\in \N^*$ is uniquely defined by its finite projections $a_k\in \N_k$ for all $k\in \ZZ$, where $a_k = \sum_{|\omega|\leq k}a(\omega)\omega$ is the orthogonal projection of $a$ onto the subspace $\N_k\subset \N$.
\begin{remark}
In many applications it is necessary to generalize to spaces built from trees with colored vertices. The theory extends from the above presentation with only minor modifications. Let $\C$ be a (finite or infinite) set of colors. A coloring of a tree or a forest is a map from its vertices to $\mathcal{C}$. Let  $\OT_{\mathcal{C}}$ and $\OF_{\mathcal{C}}$ denote colored trees and forests. For each $c\in \C$ we have the operation $B^+_c\colon \OF_{\mathcal{C}}\rightarrow \OT_{\mathcal{C}}$ creating a tree by adding a root of color $c$ to a word.
We identify $\C\subset \OT_{\mathcal{C}}\subset \OF_{\mathcal{C}}$ as the subset of single vertex trees. In the colored context we permit more general gradings $|\cdot|$ on $\OT_{\mathcal{C}}$. We allow the assignment of arbitrary positive integer weights $|c|\in \NN$ to the single vertex trees $\C\subset \OT_{\mathcal{C}}$, extended to $\OF_{\mathcal{C}}$ by $|\tau_1\ldots\tau_k| = |\tau_1|+\cdots+|\tau_k|$ and $|B_c^+\omega| = |\omega|+|c|$. The definitions of finite and infinite linear combinations of forests $\N_\C=k\langle\OT_{\mathcal{C}}\rangle$ and $\N_C^*=k\langle\langle\OT_{\mathcal{C}}\rangle\rangle$ are similar to the uni-color case.
\end{remark}

\begin{definition}\label{grafting}
The \emph{left grafting} product $\cdot \car \cdot: \N \otimes \N \rightarrow \N$ is defined recursively as follows: let $\tau \in \OT$ and $\omega, \omega_1, \omega_2 \in \OF$. Then 
\begin{eqnarray*}
\one\car\omega &=& \omega \\
\tau\car\one &=& 0 \\
\omega\car\ab & = & B^+(\omega),\\
\tau\car(\omega_1 \omega_2) &=& (\tau\car\omega_1)\omega_2 + \omega_1(\tau\car\omega_2) \\
(\tau\omega)\car\omega_1 &=&  \tau\car(\omega\car\omega_1) - (\tau\car\omega)\car\omega_1
\end{eqnarray*}
The product is extended to all of $\N$ and $\N^*$ by linearity and projective limits. 
\end{definition}
\noindent For example, 
\[\ab\ab\car\aabb = \aabababb + 2 \aabaabbb + \aaababbb\]
Note that grafting satisfies a Leibniz rule with respect to the concatenation product. If we define $\tau\bpr \omega = \tau\omega + \tau\car\omega$, we see that $\tau\car(\omega\car\omega_1) = (\tau\bpr \omega)\car\omega_1$. More generally, $\omega_1\car(\omega_2\car\omega) = (\omega_1\bpr \omega_2)\car\omega$, where $\bpr$ is the associative product defined as follows:

\begin{definition}\label{GL-product}
The Grossman-Larson product $\bpr: \N \otimes \N \rightarrow \N$ of $\omega_1, \omega_2 \in \OF$ is defined in terms of the grafting product as: 
\begin{equation*}
B^+(\omega_1 \bpr \omega_2) = \omega_1\car B^+(\omega_2),
\end{equation*}
 and is extended by linearity.
\end{definition}
It is clear that if we write $\omega_1[\omega_2]$ for $\omega_1 \car \omega_2$, we have the following structure on $\N$:
\begin{definition}[{\cite{munthe-kaas2008oth}}]\label{d-alg}
Let $A$ be a unital associative algebra with product $f,g \mapsto fg$, unit $\one$ and equipped with a non-associative composition $(.)[.]: A \otimes A \rightarrow A$ such that $\one[g] = g$ for all $g\in A$. Write $\mathcal{D}(A)$ for the set of all $f \in A$ such that $f[\cdot]$ is a derivation:
\[\mathcal{D}(A) = \{f\in A \,\,|\,\, f[gh] = (f[g])h + g(f[h]) \hspace{0.2cm} \text{for all } g,h \in A\}.\] Then $A$ is called a \emph{D-algebra} if for any derivation $f \in \mathcal{D}(A)$ and any $g\in A$ we have
\begin{eqnarray*}
&\text{(i)}& g[f] \in \mathcal{D}(A)\\
&\text{(ii)}& f[g[h]] = (fg)[h] + (f[g])[h].
\end{eqnarray*}
\end{definition}

\paragraph{The free D-algebra.}
We note that a morphism $\F: A \rightarrow A'$ of D-algebras is an algebra morphism satisfying $\F(D(\A)) \subset D(\A')$ and $\F(a[b]) = \F(a)[\F(b)]$ for all $a,b \in A.$ The D-algebra $\N$ plays a special role: it is a universal object. 

\begin{proposition}[{\cite{munthe-kaas2008oth}}]\label{universalD-alg}Let $\OT$ be planar trees decorated with colors $\C$. The vector space $\N = \RR\langle\OT\rangle$ is a free D-algebra over $\C$. That is, for any $D$-algebra $\A$ and any map $\nu\colon\C\rightarrow D(\A)$ there exists a unique D-algebra homomorphism $\F_\nu\colon N\rightarrow \A$ such that $\F_\nu(c)  = \nu(c)$ for all $c\in \C$.
\begin{diagram}[labelstyle=\scriptstyle]
\C &\rInto& \N \\
\dTo^{\nu} && \dTo_{\exists\,! \,\, \F_{\nu}} \\
D(\A) &\rInto& \A
\end{diagram}
\end{proposition}
We will see that based on this result we can construct elementary differentials and Lie--Butcher series for Lie group integrators, and also define the substitution law. To achieve this we utilize D-algebra structure of differential operators on manifolds \cite{munthe-kaas2008oth}.

\paragraph{The D-algebra of differential operators.}\label{Dalg_diffop} There is a D-algebra based on the space of vector fields\footnote{The vector fields are interpreted as differential operators acting on functions.} on the manifold $\M$. Consider the space $\mathcal{C}^{\infty} (M, \g) =: \g^{\M},$ where $\g\subset \XM$ is a Lie sub-algebra of the set of all vector fields on $\M$. For $\Psi \in \g^{\M}$ and $V\in \g$, the \emph{Lie derivative} $V[\Psi] \in \g^{\M}$ of  $\Psi$ along $V$ defined by 
\begin{equation}
V[\Psi](p) := \frac{d}{dt}\left| \right._{t=0} \Psi(\exp(tV)(p)) .
\end{equation}
$V[\cdot]$ is a first order differential operator on $\g^\M$, satisfying $V[h\Psi] = V[h]\Psi+hV[\Psi]$, where $h\in\FM$ is a scalar function\footnote{This is true when $\g^\M$ is replaced by  $\Xi^\M$ for any vector space $\Xi$.}. The Lie derivative gives rise to differential operators of higher degrees through concatenation: the concatenation of $V,W \in \g$ is a second-degree differential operator defined by $VW[\Psi] := V[W[\Psi]].$ The $\FM$-module of all differential operators, including the ones of higher degree, and the degree zero operator spanned by the identity operator $\one$, is called the \emph{universal enveloping algebra} $U(\g)$ of $\g$. We extend the structure to the space  $\mathcal{C}^{\infty} (\M, U(\g)) =: U(\g)^{\M}$ as follows: for $f,g \in U(\g)^{\M}$, $f[g] \in U(\g)^{\M}$ is defined by 
\begin{equation}\label{eq:vf_comp}
f[g](p) := (f(p)[g])(p)
\end{equation}
 and $fg \in U(\g)^{\M}$ is defined as 
 \begin{equation}\label{eq:vf_frozencomp}
fg(p) := f(p)g(p).
 \end{equation}
The latter operation is called the \emph{frozen composition} of $f$ and $g$. 
For two vector fields $f$ and $g$ written in terms of the standard coordinate frame $\{\partial / \partial x_i\}$, the operations take the following form: 
\begin{eqnarray}
f[g] &=& \sum_{i,j} f_j \frac{\partial g_i}{\partial x_j}\frac{\partial}{\partial x_i}\\
fg &=& \sum_{i,j} f_ig_j\frac{\partial}{\partial x_i}\frac{\partial}{\partial x_j}
\end{eqnarray}

The operations (\ref{eq:vf_comp}) and (\ref{eq:vf_frozencomp}) endows the space $U(\g)^{\M}$ with the structure of a $D$-algebra, where the derivations are the vector fields in $\g^\M$:
\begin{lemma}
Let $f \in \g^{\M}$ and $g, h \in U(\g)^{\M}$. Then
\begin{eqnarray*}
f[gh] &=& f[g]h + g(f[h]) \\
f[g[h]] &=& (fg)[h] + f[g][h].
\end{eqnarray*}
Hence $U(\g)^{\M}$ is a D-algebra.
\end{lemma}
\noindent The composition of $f$ and $g$ as differential operators, defined by $(f \bpr g)[h] := f[g[h]]$, is called \emph{non-frozen composition}.\footnote{We note that the two operations $f,g \mapsto f[g]$ and $f,g \mapsto f \bpr g$ gives $U(\g)^\M$ the structure of a unital dipterous algebra (as defined in \cite{loday2010cha}).}

\paragraph{Post-Lie algebras.} The theory of Lie--Butcher series can be reformulated in terms of \emph{post-Lie algebras}. These were first studied in the setting of \emph{operads} by Vallette \cite{vallette2007hog}, and also by the authors in \cite{munthe-kaas2012opl}. Our main motivation for the construction of post-Lie algebras was their relation to the D-algebras defined above, which are universal enveloping algebras of post-Lie algebras.

\subsection{Lie--Butcher series}

\paragraph{Classical B-series.} Recall (see e.g. \cite{hairer2006gni}) that a B-series is a (formal) series indexed over the set $\NT$ of \emph{non-planar} rooted trees (i.e.\ trees without any ordering of the branches) and can for a vector field $f$ be written as 
\begin{equation}\label{eq:bserclassic}
B_{h,f}(a)(y) = a(\one)y + \sum_{\tau \in \NT} \frac{h^{|\tau|}}{\sigma(\tau)} a(\tau) \F_f(\tau)(y).
\end{equation}
Here $\sigma(\tau)$ is the symmetry factor for $\tau \in \NT$, and $a$ is a map $a: \NT \rightarrow \mathbb{R}$. The map $\F_f(\tau): \mathbb{R}^n \rightarrow \mathbb{R}^n$ is the \emph{elementary differential} of the tree $\tau$, obtained recursively from $f$ and its derivatives:
\begin{equation}\label{eg:elmdiffB} \F_f(\ab)(y) = f(y), \hspace{1cm} \F_f(\tau)(y) = f^{(m)}(y)(\F_f(\tau_1)(y), ..., \F_f(\tau_m)(y)),\end{equation} where $\tau = B^+(\tau_1, \dots \tau_m)$ and $f^{(m)}$ is the $m$th derivative of the vector field. The parameter $h$ represents the step-size of the numerical method giving rise to the B-series.

\paragraph{LB-series.} We will consider a more general setting: that of differential equations evolving on manifolds. Let $\M$ be a manifold and $\XM$ the Lie algebra of vector fields $F: \M \rightarrow T\M$ on $\M$. The fundamental assumption for numerical Lie group integrators is the existence of a \emph{frame} on $TM$, defined as a finite number of vector fields $\{E_1, E_2, \dots, E_m\}$ spanning the tangent space $T_p\M$ at each point $p \in \M$. The frame is allowed to be overdetermined. It generates a Lie algebra $\g$, and it is assumed that flows of vector fields in $\g$ can be computed exactly \cite{munthe-kaas1995lbt, owren1999rkm}.  Any vector field $F: \M \rightarrow T\M$ can be written as $F(p) = \sum a_i E_i(p)$. We will study vector fields of the form $F(p) = \sum f_i(p) E_i(p)$ where $f_i: \M \rightarrow \mathbb{R}$ are smooth functions. Given such a vector field, let $f\in \g^\M$ be defined as $f_p = \sum f_i(p) E_i$. We say that $f_p$ has coefficients frozen relative to the frame. In other words, to each such $F\in \XM$ there is an associated $f\in \g^\M$ so that $F(p) = f_p\cdot p$, where $f_p\cdot p$ denotes evaluation of $f_p$ in $p$. We will often refer to such $f \in \g^\M$ as vector fields.

The general differential equation (\ref{diffeq}) can now be written as 
\begin{equation}\label{Lie-diffeq}
y' = f_y\cdot y, \hspace{0.5cm} \text{where } f\in \g^\M.
\end{equation}
The Lie--Butcher series are expansions over $\OT$ associated to integrators of this equation, just as B-series are associated to differential equations expressing the flow of vector fields in $\mathbb{R}^n$. The non-commutativity of combining vector fields is reflected in the planarity of the trees in $\OT$. 

Now we can construct the elementary differentials needed to define Lie--Butcher series. As in the classical case they can be expressed recursively by a function $\F$ based on trees. 

\begin{definition}\label{elementdiff}
The elementary differentials associated to a vector field $f: M \rightarrow \g$ is the D-algebra morphism $\F_f: \N \rightarrow U(\g)^{\M}$ we get from Proposition \ref{universalD-alg} by associating the tree $\ab$ to $f$ (i.e.\ $\C = \{\ab\}$ and $\nu: \ab \mapsto f$ in Proposition \ref{universalD-alg}). Hence $\F_f$ is defined by 
\begin{itemize}
\item[(i)] $\F_f(\mathbb{I}) = \mathbb{I}$
\item[(ii)] $\F_f(B^+(\omega)) = \F_f(\omega)[f]$
\item[(iii)] $\F_f(\omega_1 \omega_2) = \F_f(\omega_1)\F_f(\omega_2)$
\end{itemize}
\end{definition}
\noindent When the vector field $f$ is clear from the context we will occasionally write $\F$ instead of $\F_f$. 
\begin{definition}~\label{lbseries}
For an infinite series $\alpha\in \N^*=\RR\langle\langle\OT\rangle\rangle$ a \emph{Lie--Butcher series} is a formal series in $U(\g)^\M$ defined as\[\Bs_{f}(\alpha) = \sum_{\omega\in\OF}\alpha(\omega) \F(\omega). \]
\end{definition}
\noindent For a vector field $f$ this can also be written as the commutative diagram
\begin{diagram}[labelstyle=\scriptstyle]
\{\ab\} &\rInto& \N^* \\
\dTo^f && \dTo_{\Bs_f} \\
\g^M &\rInto& U(\g)^M
\end{diagram}
where $\Bs_f$ is the unique D-algebra homomorphism given by Proposition \ref{universalD-alg}.

\begin{remark}
By coloring the vertices of the trees vi a map $\nu$ we can define $\mathcal{F}$ and $\Bs$ for multiple vector fields. The elementary differentials $\mathcal{F}_{\nu}$ are still obtained from Proposition \ref{universalD-alg}, but the set $\mathcal{C}$ will contain multiple colors.
\end{remark}

\subsection{Some algebraic constructions}\label{sect:algebra}

Before we show how LB-series can be used to represent flows of vector fields on manifolds we must conduct a closer study of the space where the coefficients $\alpha$ live. To understand the various ways we can represent such flows it will also be helpful to look at some Lie idempotents, namely the eulerian and Dynkin idempotents (Section \ref{lieidem}), and also certain non-commutative polynomials called Bell polynomials (Section \ref{bell}). We will follow the presentation in \cite{munthe-kaas2008oth} and \cite{lundervold2009hao}.

\subsubsection{The Hopf algebras $\mathcal{H}_{Sh}$ and $\Hn$}\label{hopfalgebras}

It is well known that inserting a B-series $B_{h,f}(a)$ into another series $B_{h,f}(b)$ results in a B-series $B_{h,f}(a)(B_{h,f}(b)(y)) = B_{h,f}(a\cdot b)(y)$. The product $a\cdot b$ on the set of maps $a: \OT \rightarrow \RR$ with $a(\one)=1$ gives rise to a group, called the \emph{Butcher group} \cite{butcher1972aat, hairer1974otb}. This is the group of characters in a variant of the Connes--Kreimer Hopf algebra of renormalization \cite{connes1998har, brouder2000rkm}. A similar result holds for LB-series, where the Hopf algebra of Connes--Kreimer is replaced by a more general Hopf algebraic structure on the set of rooted trees. This Hopf algebra was introduced in \cite{munthe-kaas2008oth}. See also \cite{lundervold2009hao, lundervold2011oas}. 

Note first that the vector space $\RN$ spanned by trees can be turned into a Hopf algebra by using concatenation as product and \emph{deshuffling} as coproduct. The deshuffling coproduct $\Delta_{Sh}$ is results from requiring the trees to be primitive, and extending by concatenation: 
\[\Delta_{Sh}(\tau) = \tau \otimes \one + \one \otimes \tau, \hspace{1cm} \Delta_{Sh}(\tau_1\tau_2) = \Delta_{Sh}(\tau_1)\Delta_{Sh}(\tau_2),\]
where $\tau$, $\tau_1$, $\tau_2$ are trees. The antipode is defined by 
\[S(\tau_1\tau_2\cdots \tau_n) = (-1)^n\tau_n\tau_{n-1}\cdots \tau_1,\]
and the unit $\eta$ and counit $\epsilon$ by $\eta(1) = \one$, and $\epsilon(\one) = 1$, $\epsilon(\omega) = 0$, for all $\omega \in \OF \setminus \one$.

The vector space $\N = \RN$ can also be turned into an algebra using the \emph{shuffle product} $\shuffle: \N \otimes \N \rightarrow \N$ defined recursively by $\one \shuffle \omega = \omega = \omega \shuffle \one$ and 
\begin{equation}
(\tau_1\omega_1) \shuffle (\tau_2 \omega_2) = \tau_1(\omega_1 \shuffle \tau_2 \omega_2) + \tau_2(\tau_1\omega_1 \shuffle \omega_2)
\end{equation}
for $\tau_1, \tau_2 \in \OT$, $\omega_1, \omega_2 \in
\OF$.\footnote{Coproducts will occasionally be written using the
  \emph{Sweedler} notation $\Delta(\omega)=\sum \omega_{(1)} \otimes
  \omega_{(2)}$.} This algebra can be given the structure of a bialgebra in several different ways. We can equip it with the coproduct $\Delta_c: \N \rightarrow \N \otimes \N$ given by \emph{deconcatenation} of words: 
\begin{equation}
\Delta_c(w) = \one \otimes w + w \otimes \one + \sum_{i=1}^{n-1} \tau_1 \cdots \tau_i \otimes \tau_ {i+1} \cdots \tau_n,
\end{equation}
where $\omega = \tau_1 \cdots \tau_n$. This results in the shuffle bialgebra, which equipped with the same antipode, unit and counit as the deshuffle Hopf algebra defines the \emph{shuffle Hopf algebra} $\mathcal{H}_{Sh}.$\footnote{Note that the concatenation deshuffling Hopf algebra is dual to the shuffle deconcatenation Hopf algebra.} 

If we instead equip $\N$ with the coproduct $\Delta_N: \N \rightarrow \N \otimes \N$ defined recursively as $\Delta_N(\one) = \one \otimes \one$ and 
\begin{equation}
\Delta_N(\omega \tau) = \omega \tau \otimes \one + \Delta_N(\omega) \shuffle \, \cdot \, (I \otimes B^+) \Delta_N(B^-(\tau)),
\end{equation}
 where $\tau \in \OT$, $\omega \in \OF$, we get another bialgebra
 $\Hn$. Here $\shuffle \,\cdot: \N^{\otimes 4} \rightarrow \N \otimes
 \N$ denotes shuffle on the left and concatenation on the right:
 $(\omega_1 \otimes \omega_2) \shuffle \, \cdot\, (\omega_3 \otimes
 \omega_4) = (\omega_1 \shuffle \omega_3) \otimes (\omega_2
 \omega_4).$ An explicit description of the coproduct in terms of tree
 cuts can be found in Section \ref{substlov} below, and in
 \cite{munthe-kaas2008oth}, where it was shown that $\Delta_N$ is the
 dual of the Grossman-Larson product and that $\Hn$ forms a Hopf algebra.\footnote{Being a graded and connected bialgebra $\Hn$ is automatically a Hopf algebra. A more direct argument, and formulas for the antipode, can be found in \cite{munthe-kaas2008oth}} This is the Hopf algebra governing composition of LB-series (Theorem \ref{compositionOfLB}).

To simplify the expressions we introduce a \emph{magmatic} structure (i.e. the structure of a set equipped with a closed binary operation with no further relations) on $\OF$. Additional details and motivation for the introduction of this structure can be found in Section \ref{implement}. Let $\omega_1$, $\omega_2$ be two elements of $\OF$ and define the operation $\times: \OF \times \OF \rightarrow \OF$ by 
\begin{equation}
\omega_1 \times \omega_2 = \omega_1 B^+(\omega_2).
\end{equation}
For example, 
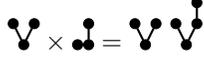
\[\aababb \times \ab\aabb =  \aababb\,\,\aabaabbb\]
This operation is magmatic, and the empty tree $\one$ freely generates all of $\OF$. The operation is extended to $\N = \RN$ via linearity. If $\omega = v_1 \times v_2$, $\omega \neq \one$, then we call $v_1$ the \emph{left part} $\omega_L$ of $\omega$ and $v_2$ the \emph{right part} $\omega_R$. The shuffle of two elements of the magma can be defined as 
\begin{equation}
v \shuffle \omega = (v \shuffle  \omega_L)\times v_R + (v_L \shuffle \omega) \times \omega_R,
\end{equation}
$\omega\shuffle \one = \one \shuffle \omega = \omega$, and we notice that the coproduct in $\Hn$ can be written as  
\begin{equation}
\Delta_N(\omega) = \omega \otimes \one + \Delta_N(\omega_L) \shuffle \times \Delta_N(\omega_R),
\end{equation}
where $\shuffle \times$ now denotes shuffle on the left, magma operation on the right.

\paragraph{Characters and the composition of LB-series.} Recall that a \emph{character} of a Hopf algebra $(H, \Delta, \cdot)$ over a field $\k$ is an algebra morphism $\alpha: H \rightarrow \k$, e.g. $\alpha(a\cdot b) = \alpha(a)\alpha(b),$ and $\alpha(1_H) = 1$, where $a,b \in H$, and $1_H$ denotes the unit. The convolution product $\alpha * \beta$ of two characters is defined by \[H \overset{\Delta}{\longrightarrow} H \otimes H \overset{\alpha \otimes \beta}{\longrightarrow} \k \otimes \k \rightarrow \k.\] This gives the set of characters $G(H)$ of $H$ the structure of a group. In fact, the field $\k$ can be replaced by any commutative algebra $A$, giving rise to \emph{$A$-valued characters}. Another type of character we will need later are the \emph{infinitesimal characters}. An $A$-valued infinitesimal character is a linear map $\alpha: H \rightarrow A$ satisfying  
\begin{equation}
\alpha(h\cdot h') = \mu_A(\alpha(h), \delta(h')) + \mu_A(\delta(h), \alpha(h')),
\end{equation}
where $\mu_A$ is the product in $A$ and $\delta$ is the composition of the counit of $H$ and the unit of $A$, $\delta = \eta_A \circ \epsilon$. The characters and infinitesimal characters are connected via the exponential and logarithm, see e.g. \cite{manchon2008hai}. 

The group structure of the characters in $\Hn$ exactly corresponds to the composition of LB-series.

\begin{theorem}[{\cite{munthe-kaas2008oth}}]\label{compositionOfLB} The composition of two LB-series is again a LB-series: 
\[\Bs_f(\alpha)[\Bs_f(\beta)] = \Bs_f(\alpha * \beta),\] 
where $*$ is the convolution product in $\Hn$.
\end{theorem}

\subsubsection{Lie idempotents}\label{lieidem}
A \emph{Lie polynomial} over an algebra $A$ is an element of the smallest submodule of $\RR\langle A \rangle$ that is closed under the bracket $[P,Q] := PQ-QP$ in $\RR\langle A \rangle$. The Lie algebra of these polynomials is the free Lie algebra $\Lie(A)$ on $A$ \cite{reutenauer93fla}. There are several important idempotent maps, called \emph{Lie idempotents}, from $\RR\langle A \rangle$ to $\Lie(A)$.

\paragraph{Eulerian idempotent}
Let $H$ be a commutative, connected and graded Hopf algebra. Consider $\End_k(H) = \Hom_k(H,H)$ equipped with the convolution product $\ast$. Let $\id\in\End_{k}(H)$ be the identity endomorphism and $\delta = \eta\circ \epsilon\in\End_{k}(H)$ the unit of convolution.

\begin{definition}[{\cite{loday97ch}}]
The Eulerian idempotent $e\in \End(H)$ is given by the formal power series \[e := \log^*(\Id) = J - \frac{J^{*2}}{2} + \frac{J^{*3}}{3} + \cdots (-1)^{i+1} \frac{J^{*i}}{i} + \cdots,\] where $J=\Id-\delta.$
\end{definition}

\begin{proposition}[{\cite{loday97ch}}]
For any commutative graded Hopf algebra $H$, the element $e \in \End_k(H)$ defined above is a Lie idempotent in $H$. That is, $e \circ e = e$ and it has image in the free Lie algebra.
\end{proposition}
\noindent The practical importance of the Eulerian idempotent in numerical analysis arises in the study of backward error analysis, where the following lemma provides a computational formula for the logarithm:

\begin{proposition}[{\cite{lundervold2009hao}}]\label{eulerlog} For $\alpha\in G(H)$ and $h\in H$, we have
\[\log^\ast(\alpha)(h) = \alpha(e(h)).\]
In other words, the logarithm can be written as right composition with the eulerian idempotent: \[\log^{\ast} = \_ \circ e: G(H) \rightarrow \g(H).\]
\end{proposition}

\paragraph{Dynkin idempotent}

The classical \emph{Dynkin operator} on the shuffle Hopf algebra is given by left-to-right bracketing: \[D(a_1...a_n) = [\dots[[a_{1},a_{2}], a_{3}], \dots, a_{n}], \quad\mbox{where $[a_i,a_j] = a_ia_j-a_ja_i$}.\]  Letting $Y(\omega) = \#(\omega)\omega$ denote the grading operator, where $\#(\omega)$ is word length, it is known that the \emph{Dynkin idempotent} $Y^{-1}\circ D$ is an idempotent projection on $\Lie(A)$. As in \cite{ebrahimi-fard07alt}, the Dynkin operator can be written as the convolution of the antipode $S$ and the grading operator $Y$: $D = S\ast Y$. This description can be generalized to any graded, connected and commutative Hopf algebra $H$:

\begin{definition}
Let $H$ be a graded, commutative and connected Hopf algebra with grading operator $Y: H \rightarrow H$.  The \emph{Dynkin operator} is the map $D: H \rightarrow H$ given as \[D := S * Y.\]
\end{definition}

\subsubsection{The non-commutative Bell polynomials}\label{bell}

In \cite{munthe-kaas1998rkm} some non-commutative polynomials $B_n$ were introduced to express the Butcher order theory of Runge-Kutta methods on manifolds. In \cite{lundervold2009hao} it was observed that these polynomials were a non-commutative analogue of Bell polynomials, and that they could be used to study more general flows on manifolds. We recall their definition here. 

Let $\I=\{d_j\}_{j=1}^\infty$ be an infinite alphabet in 1--1 correspondence with $\NN^+$, and consider the 
free associative algebra $\Hfdb = \RR\langle \I \rangle$ with the grading given by $|d_j| = j$ and $|d_{j_1}\cdots d_{j_k}|= j_1+\cdots+j_k$. Let $\partial\colon \Hfdb\rightarrow \Hfdb$ be the derivation given by $\partial(d_i) = d_{i+1}$, linearity and the Leibniz rule
$\partial(\omega_1\omega_2) = \partial(\omega_1)\omega_2 + \omega_1\partial(\omega_2)$ for all $\omega_1,\omega_2\in\I^*$. We let $\#(\omega)$ denote the length of the word $\omega$.

\begin{definition}\label{def:bell}The non-commutative Bell polynomials $B_n := B_n(d_1,\ldots, d_n)\in \Hfdb$ are defined by the recursion
\begin{eqnarray*}
B_0 & = & \one\\
B_n& =& (d_1+\partial)B_{n-1} = (d_1+\partial)^n \one\quad \mbox{for $n>0$}.
\end{eqnarray*}
\end{definition}
\noindent The first few are
 \begin{eqnarray*}
B_0 &=& \one\\
B_1 & = & d_1\\
B_2 & = & d_1^2 + d_2 \\
B_3 &=& d_1^3 + 2d_1d_2 + d_2d_1+d_3\\
B_4 & = & d_1^4+ 3d_1^2 d_2 +2d_1d_2d_1 + d_2d_1^2 +3d_1d_3+  d_3d_1 + 3d_2d_2 + d_4.
\end{eqnarray*}
We write $B_{n,k} := B_{n,k}(d_1,\dots, d_{n-k+1})$ for the part of $B_n$ consisting of the words of length $k$, e.g. $B_{4,3}=3d_1^2d_2 + 2d_1d_2d_1 + d_2d_1^2$. It is often useful to employ the polynomials  $Q_{n}$ and $Q_{n,k}$ related to  $B_{n}$ and $B_{n,k}$ by the following rescaling:
\begin{eqnarray}
Q_{n,k}(d_1,\ldots,d_{n-k+1}) &=& \frac{1}{n!}B_{n,k}(1!d_1,\ldots,j!d_j,\ldots) = \mathop{\sum}_{|\omega| = n, \#(\omega)=k} \kappa({\omega})
\omega\\
Q_n(d_1,\ldots,d_n)&=&\sum_{k=1}^n Q_{n,k}(d_1,\ldots,d_{n-k+1})\\
Q_0 &:=& \one.
\end{eqnarray}
\noindent These polynomials can be used to define an operator $Q$ on any graded Hopf algebra $H$. Let $d_i$ be defined on $H^*$ by 
\begin{equation}\label{eq:Q}
d_i(\alpha) = \alpha_i = \left.\alpha\right|_{H_i}, \qquad d_id_j(\alpha) = \alpha_i\ast\alpha_j,
\end{equation}
where $\alpha_i$ is the degree $i$ component of $\alpha$ and $*$ is the convolution product. The operator $Q$ is a bijection from infinitesimal characters to characters of $H$ (for details, see \cite{lundervold2009hao}).

\subsection{Lie--Butcher series and flows of vector fields}\label{flows}
Flows $y_0 \mapsto y(t) = \Psi_t(y_0)$ on the manifold $\M$ can be represented by LB-series in several different ways. Here are three procedures, giving rise to what we will call LB-series of Type 1, 2 and 3:
\begin{enumerate}
\item In terms of pullback series: Find $\alpha\in G(\Hn)$ such that 
  \begin{equation}\label{LBpullback}
\Psi(y(t)) = \Bs(\alpha)(y_0)[\Psi] \quad\mbox{for any $\Psi\in U(\g)^\M$.}
\end{equation}
This representation is used in the analysis of Crouch--Grossman methods by Owren and Marthinsen~\cite{owren1999rkm}.
In the classical setting this is called a $S$-series~\cite{murua1999fsa}.
\item In terms of an autonomous differential equation: Find $\beta\in \g(\Hn)$ such that $y(t)$ solves
\begin{equation}
y'(t) = \Bs(\beta)(y(t)).
\end{equation}
This is called backward error analysis (confer Section \ref{bea}).
\item In terms of a non-autonomous equation of \emph{Lie type} (time dependent frozen vector field): Find $\gamma\in \g(\Hsh)$ such that $y(t)$ solves
\begin{equation}\label{lietype}
y'(t) = \left(\frac{\partial}{\partial t} \Bs(\gamma)(y_0)\right) y(t).
\end{equation}
This representation is used in~\cite{munthe-kaas1995lbt,munthe-kaas1998rkm}. In the classical setting this is (close to) the standard definition of $B$-series. 
\end{enumerate}
The algebraic relationships between the coefficients $\alpha$, $\beta$
and $\gamma$ in the above LB-series are \cite{lundervold2009hao}:
\begin{align*}
\beta &= \alpha\opr e &\mbox{$e$ is Euler idempotent in $\Hn$. (Proposition \ref{eulerlog})}\\
\alpha &= \exp^\bpr(\beta)&\mbox{Exponential wrt.\ GL-product}\\
\gamma &= \alpha\opr Y^{-1}\opr D &\mbox{Dynkin idempotent in $\Hsh$. \cite[Proposition 4.4]{lundervold2009hao}}\\
\alpha &= Q(\gamma)&\mbox{$Q$-operator~(\ref{eq:Q}) in $\Hsh$. \cite[Proposition 4.9]{lundervold2009hao}}
\end{align*}
By using these relationships one can convert between the various representations of flows.

In the notation in the following examples of LB-series we suppress the vector fields and elementary differentials, and phrase the LB-series in terms of the (dual of the) coefficient functions.

\begin{example}[The exact solution]
The exact solution of a differential equation 
\[y'(t) = F(y(t))\] 
can be written as the solution of
\[y' = F_t\dpr y, \quad y(0)= y_0,  \]
where $F_t=F(y(t))\in \g$ is the pullback of $F$ along the time dependent flow of $F$. Let $F_t = \frac{\partial}{\partial t}\Bs_t(\gamma)$. By \cite[Proposition 4.9]{lundervold2009hao} the pullback is given by $\Bs_t(Q(\gamma_{\text{Exact}}))[F]$, so 
\[Y\opr\gamma_{\text{Exact}} = Q(\gamma_{\text{Exact}})[\ab]\Rightarrow \gamma_{\text{Exact}} = Y^{-1}\opr B^+(Q(\gamma_{\text{Exact}})).\]
Note that this is reminiscent of a so-called combinatorial Dyson--Schwinger equation~\cite{foissy2008fdb}. Solving by iteration yields
\begin{eqnarray*}
\gamma_{\text{Exact}} & = &
\ab + \frac{1}{2!}\aabb + \frac{1}{3!}\left(\aababb+\aaabbb\right) + \frac{1}{4!}\left(\aabababb+\aaabbabb+2\aabaabbb
+\aaababbb+\aaaabbbb\right)+ \frac{1}{5!}(\aababababb+\aaabbababb\\
& &
+2\aabaabbabb +3\aababaabbb+\aaababbabb+ \aaaabbbabb+3\aaabbaabbb+3\aabaababbb+3\aabaaabbbb+\aaabababbb+\aaaabbabbb+2\aaabaabbbb+\aaaababbbb+\aaaaabbbbb
)+\\
& &
\frac{1}{6!}\left(\aabababababb+\cdots\right)+\cdots
\end{eqnarray*}
Note that a formula for the LB-series for the exact solution was given in \cite{owren1999rkm}. We observe that there cannot be any commutators of trees in this expression. Therefore, in LB-series of numerical integrators, commutators of trees must be zero up to the order of the method.
\end{example}

\begin{example}[The exponential Euler method]\label{eulerMethod}
The exponential Euler method \cite{iserles2000lgm} can be written as follows:
\[y_{n+1} = \exp(hf(y_n))y_n,\] or, by rescaling the vector field $f$, as \[y_{n+1} = \exp(f(y_n))y_n.\] This equation can be interpreted as a pullback equation of the form $\Phi(y_{n+1}) = \Bs(\exp(\ab))[\Phi]y_n$, so 
\[\alpha = \exp(\ab) = \one + \ab + \frac{1}{2!}\ab\ab + \frac{1}{3!}\ab\ab\ab + \cdots.\]
(Here the Grossman-Larson product is the same as concatenation). Note that $\exp(\ab) = Q(\ab),$ so the Type 3 LB-series for the Euler method is simply \[\gamma_{\text{Euler}} = \ab.\]
\end{example}

\begin{example}[The implicit midpoint method]\label{midpointMethod}
The implicit midpoint method \cite{iserles2000lgm} can be presented as:
\begin{eqnarray}\label{midpoint}
\sigma &=& f(\exp(\frac{1}{2}\sigma)y_n) \\
y_{n+1} &=& \exp(\sigma)y_n
\end{eqnarray}
We make the following ansatz: 
\begin{equation}\label{mp:ansatz}
\sigma = \sum_{\omega} \alpha(\omega)\omega = \alpha(\ab)\ab + \alpha\left(\aabb\right)\aabb + \alpha\left(\aababb\right)\aababb + \alpha\left([\ab,\aabb]\right)[\ab,\aabb] + \alpha\left(\aaabbb\right)\aaabbb + \cdots,
\end{equation}
i.e. that $\sigma$ can be written as an infinitesimal LB-series. From Equation \ref{midpoint}, we get that
\begin{equation}\label{mp:sigma}
\sigma = \sum_{j=0}^{\infty} \frac{(\sigma)^j}{2^jj!} [\ab].
\end{equation}
Since there are no forests in this expression, we must have $\alpha([\omega, \omega']) = 0$ for all $\omega, \omega' \in \OT$. If we write $\tau= B^+(\tau_1 \cdots \tau_j)$, then by combining Equation \ref{mp:sigma} with the ansatz, we see that coefficients of the LB-series are given recursively as $\alpha(\ab) = \frac{1}{2}$, 
\begin{equation}
\alpha(\tau) = \frac{1}{2^jj!} \alpha(\tau_1) \cdots \alpha(\tau_j).
\end{equation}
Hence
\begin{eqnarray*}
\alpha_{\text{Midpoint}} & = &
\ab + \frac{1}{2!}\aabb + \frac{1}{2}\left(\frac{1}{4}\aababb+\aaabbb\right) +\cdots
\end{eqnarray*}
\end{example}

\section{Substitution law for Lie--Butcher-series}
In this section we will generalize the substitution law for B-series \cite{chartier2005asl} to LB-series. Once the substitution law has been established we will apply it to backward error analysis for numerical methods based on LB-series. 

\subsection{The substitution law}\label{substlov}
Consider $\N$ as a D-algebra where the derivations are the Lie polynomials $D(\N) = \g(\Hn)\cap \N$. By the universal property of $\N$, we know that for any map  $a\colon \C\rightarrow D(\N)$ there exists a unique D-algebra homomorphism $\F_a:\N\rightarrow \N$ such that $\F_a(c)=a(c)$ for all $a\in \C$. This is called the substitution law.

\begin{multicols}{2}
\begin{definition}\label{substlaw}For any map  $a: \C \rightarrow D(\N)$ the unique D-algebra homomorphism $a\star:\N\rightarrow \N$ such that $a(c) = a\star c$ for all $c\in \C$ is called $a$-substitution\footnote{In most applications we want to substitute infinite series and extend $a\star$ to a homomorphism $a\star\colon\N^*\rightarrow \N^*$. 
The extension to infinite substitution is straightforward because of the grading, we omit details. We write $a\star$ also for infinite substitution.}.
\newline
\begin{diagram}[labelstyle=\scriptstyle]
\C &\rInto& \N\\
\dTo^a && \dTo_{a\star} \\
D(\N) &\rInto& \N.
\end{diagram}
\end{definition}
\end{multicols}

\begin{theorem}\label{substBseries}
The substitution law defined in Definition \ref{substlaw} corresponds to the substitution of B-series in the sense that 
\[\Bs_{\Bs_f(\beta)}(\alpha) = \Bs_f(\beta \star \alpha)\]
\end{theorem}
\noindent The theorem is easily proven by using the following lemma:
\begin{lemma}\label{substBserieslemma}
For all $\beta: \{\ab\} \rightarrow D(\N^*)$ and all B-series $B_f: \N^* \rightarrow U(\g)^{\M},$ the composition $B_f \circ \beta$ has image in $\g^{\M}$. In other words, B-series maps $D(\N)$ to derivations on $\M$. 
\end{lemma}
\begin{proof}
It is enough to prove this for Lie polynomials \footnote{Lie series are formal series whose homogeneous components are Lie polynomials \cite{reutenauer93fla}}. Since $B_f$ is a D-algebra homomorphism it maps trees to derivations, so the only thing we have to check is that the commutator $[V,W]=VW-WV$ of two derivations $V$ and $W$ is a derivation. This is a straightforward calculation.
\end{proof}

\begin{proof}[Proof of Theorem \ref{substBseries}]
Except for the use of Lemma \ref{substBserieslemma}, the proof is purely categorical. Let $B_f$ be a B-series. The composition of $B_f$ with the map $\beta \star$ can be written in diagrammatic form as
\begin{diagram}
\{\ab\} &&\rInto&& \N^* \\
             & \rdTo^{\beta}&           &           &           \\
            &           & D(\N) &           &  \dTo_{\beta\star}\\
             &           &           & \rdInto\\
 &            & &            & \N^* \\
             &           &  &            & \dTo_{B_f} \\
            &            &           &             & U(\g)^M
\end{diagram}
By Lemma \ref{substBserieslemma} the composition of the two diagonal arrows and $B_f$ actually has image in $\g^{\M}$. Therefore the universal property for the diagram obtained by adding the map $\Bs_f \circ \beta: \{\ab\} \rightarrow \g^M$ to the above diagram shows that $\Bs_f \circ \beta\star = \Bs_{\Bs_f \circ \beta}$, and hence the theorem.
\end{proof}

Many of the useful properties of the substitution law follow immediately from the fact that $a \star$ is a D-algebra homomorphism. For example, $a\star: \N\rightarrow \N$ is a linear map which for any $n,n'\in\N$ satisfies
\begin{align*}
a\star\one &= \one\\
a\star (nn') &= (a\star n)(a\star n')\\
a\star (n\car n') &= (a\star n)\car (a\star n')\\
a\star (n\bpr n') &= (a\star n)\bpr (a\star n')\\
a\star(n\opr S) &= (a\star n)\opr S\\
a\star(n\opr e) &= (a\star n)\opr e\\
\end{align*}
where $S$ is the antipode and $e$ is Euler map in $\Hn$.

The free D-algebra $\N$ is the universal enveloping algebra of the free \emph{post-Lie algebra} $\g$ of rooted trees \cite{munthe-kaas2012opl}. By defining a coproduct by requiring that the elements of $\g$ are primitive (e.g. the deshuffle coproduct of Section \ref{hopfalgebras}), it is a bialgebra. The unique D-algebra morphism $a\star$ is a coalgebra morphism for this coproduct:
\begin{lemma}\label{coalg_subst}
The map $a\star$ is a coalgebra morphism with respect to the coproduct given by deshuffling of words (Section \ref{hopfalgebras}). That is, 
\[(a\star \otimes a\star) \circ \Delta_{Sh} = \Delta_{Sh} \circ a\star,\]
where $\Delta_{Sh}$ denotes the deshuffling coproduct. 
\end{lemma}
\begin{proof}
The result is easily proven for primitive elements. The general case follows by induction on the length of words.
\end{proof}

\begin{remark}[The Hopf algebra for the substitution law]
Based on the results in \cite{calaque2009tih} and the fact that the operad governing post-Lie algebras is known, it is possible to describe the Hopf algebra for the substitution law following the program in \cite{calaque2009tih}. This is a project currently under development \cite{ebrahimi-fard2011otp}.
\end{remark}

\subsection{A formula for the substitution law} The substitution law can be calculated recursively using a formula involving trees. To write down the formula we need to look at cutting operations on trees and forests. 

\paragraph{Cutting trees and forest.} Let $\tau \in \OT$ be an ordered rooted tree. An \emph{elementary left cut} $c$ of $\tau$ is a choice of a set of branches $E$ of $\tau$ to be removed from $\tau$. These are chosen in a systematic manner: if an edge $e$ is in $E$ then all the branches on the same level and to the left of $e$  must also be in $E$. Each cut splits $\tau$ into two components: the pruned part $P^c_{el}(\tau)$ consisting of the trees that were cut off concatenated together, and the remaining part $R^c_{el}(\tau)$ consisting of the tree containing the root. We also consider the \emph{empty cut}, i.e. the cut $c$ so that $P^c_{el}(\tau) = \one$ and $R^c_{el}(\tau)=\tau$, to be an elementary cut.

\begin{table}[!ht]
  \begin{equation*}
    \begin{array}{c|cc|c} 
       \hline \\[-2mm]
        \tau = \aabaabbb&& P^c_{el}(\tau) & R^c_{el}(\tau)  \\[1mm]
        \hline \\[-2mm]
        && \one & \aabaabbb \\[2mm]
        && \ab & \aaabbb \\[2mm]
        &&\ab\,\aabb & \ab \\[2mm]
        &&\ab & \aababb
    \end{array}
  \end{equation*}
\end{table}

\noindent A \emph{left admissible cut} on $\tau$ consists of a collection of elementary cuts applied to $\tau$ with the property that any path from the root to any vertex of $\tau$ crosses at most one elementary cut. The pruned parts corresponding to each elementary cut are shuffled together, with no internal shuffling of the trees resulting from each elementary cut. An admissible cut of a tree results in a collection of shuffles of forests $P^c(\tau)$ and a tree $R^c(\tau).$ The collection of all left admissible cuts for a tree $\tau$ is written as $LAC(\tau)$.
\begin{figure}[!h]
  \begin{equation*}
    \begin{array}{c|cc|c} 
       \hline \\[-2mm]
        \tau = \aabaabbb&& P^c(\tau) & R^c(\tau)  \\[1mm]
        \hline \\[-2mm]
        && \one & \aabaabbb \\[2mm]
        && \ab & \aaabbb \\[2mm]
        &&\ab\,\aabb & \ab \\[2mm]
        &&\ab & \aababb \\[2mm]
        && \ab \shuffle \ab & \aabb
    \end{array}
  \end{equation*}

\end{figure}

\noindent We extend these cutting operations to forests $\omega \in \OF$ by applying the $B^+$ operator to $\omega$ and then cut it as a tree without using cuts of branches growing out of the root, before finally applying the $B^-$ operator to $R^c(\omega)$ to remove the added root.

The coproduct $\Delta_N$ of the Hopf algebra $\Hn$ (Section \ref{sect:algebra}) can be formulated in terms of these cuts \cite{munthe-kaas2008oth}. First one must extend the left admissible cuts to include the \emph{full} cut of a tree, which cuts ``below'' the root, so that $P^c_{el}(\tau)$ is again $\tau$. The set of all left admissible cuts, including the full cut, is denoted by FLAC, and the coproduct $\Delta_N$ can be written as:
\begin{equation}
\Delta_N(\omega) = \sum_{c \in FLAC} P^c(\omega) \otimes R^c(\omega).
\end{equation}
Table \ref{coprod} gives the result of this coproduct applied to all forests up to order 4. If we let $\tilde{\Delta}_N(\omega)$ consist only of forests resulting from not using the empty nor the full cut, we get $\Delta_N(\omega) = 1 \otimes \omega + \omega \otimes 1 + \tilde{\Delta}_N(\omega)$. The operation $\tilde{\Delta}_N$ is called the \emph{reduced coproduct}.

\paragraph{A formula for the substitution law.} We will give a formula for the dual of the substitution, i.e. a formula for $a_{\star}^T$, where $\langle a \star b, \omega \rangle = \langle b, a_{\star}^T(\omega)\rangle$. The formula is based on the pruning operation $\mathcal{P}$ on forests. 
\begin{lemma}[Pruning] Let $\omega$ and $\nu$ be two forests. The dual of grafting, i.e. the operation defined by $\langle \nu\car\omega',\omega \rangle = \langle \omega', \mathcal{P}_{\nu}(\omega) \rangle$, is given by:  
\begin{equation*}
\mathcal{P}_{\nu}(\omega) = \sum_{c \in LAC(\omega)} \langle \nu, P^c(\omega)\rangle R^c(\omega).
\end{equation*}
 The operation is called \emph{pruning}.
\end{lemma}
\begin{proof}
An elementary cut at an edge growing out of a node $n$ of a forest $\omega$ is the dual operation of attaching trees via edges to the node $n$ in a certain order, e.g. 
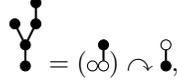
\[\aaabaabbbb = (\AB \AabB) \car \aABb,\]
where the white nodes indicates where the attachment is done. The shuffling in $P^c(\omega)$ corresponds to the dual of attaching forests in all possible ways to different nodes. Hence the dual of grafting is given by
\[\sum_{c \in LAC(\omega)} P^c(\omega) \otimes R^c(\omega).\]
\end{proof}

\begin{theorem}\label{substlawformula}
We have 
\begin{equation*}\label{eq:substlaw}
a_{\star}^T (\omega) = \sum_{(\omega) \in \Delta_c} \sum_{c \in LAC(\omega_{(2)})} a_{\star}^T(\omega_{(1)}) B^+\left(a_{\star}^T(P^c(\omega_{(2)}))\right) a(R^c(\omega_{(2)})),
\end{equation*}
if $\omega \neq \one$, and $a_{\star}^T(\one) = \one$. Here $\Delta_c$ denotes deconcatenation (Section \ref{sect:algebra}).
\end{theorem}

\noindent Note that using the magmatic product $\times$ defined in Section \ref{sect:algebra}, this can also be written as: 
\begin{equation}\label{eq:magmasubstlaw}
a_{\star}^T = \mu \circ (\mu_{\times} \otimes I)\circ (a_{\star}^T \otimes a_{\star}^T \otimes a) \circ (I\otimes \Delta'_N) \circ \Delta_c,
\end{equation}
where $\mu$ is concatenation, $\Delta_N$ is the coproduct in $\Hn$, and $\Delta'_N(\omega) = \Delta_N(\omega) - \omega \otimes \one$ for all forests $\omega$.

\begin{proof}
We first prove the formula for ordered trees. Let $\pi_{\OT}$ denote the projection of forests onto trees: $\pi_{\OT}(\omega) = \sum_{\tau \in \OT} \langle \tau, \omega \rangle \omega$. Recall that 
$$\langle a \star \omega', \omega \rangle = \langle \omega', a_{\star}^T(\omega) \rangle, \hspace{1cm} \langle \nu\car\omega',\omega \rangle = \langle \omega', \mathcal{P}_{\nu}(\omega) \rangle.$$
We have
\begin{eqnarray*}
\pi_{\OT}(a_{\star}^T \omega) = \sum_{\tau \in \OT} \left\langle \tau, a_{\star}^T \omega \right\rangle \tau 
&=& \sum_{\nu \in \OF} \left\langle \nu\car\ab, a_{\star}^T(\omega)\right\rangle \nu\car\ab \\
&=& \sum_{\nu \in \OF} \left\langle a \star (\nu\car\ab), \omega\right\rangle \nu\car\ab \\
&=& \sum_{\nu \in \OF} \left\langle (a \star\nu)\car a, \omega \right\rangle \nu\car\ab \\
&=& \sum_{\nu \in \OF} \left\langle a, \mathcal{P}_{a\star\nu} \omega \right\rangle \nu\car\ab\\
&=& \sum_{c \in LAC(\omega)} \sum_{\nu \in \OF} \left\langle a\star \nu, P^c(\omega)\right\rangle \left\langle a, R^c(\omega)\right\rangle \nu\car\ab.
\end{eqnarray*}
Hence, 
\begin{eqnarray*}
\pi_{\OT}(a_{\star}^T \omega) &=& \sum_{c \in LAC(\omega)} \sum_{\nu \in \OF} \left\langle \nu, a_{\star}^T P^c(\omega) \right\rangle (\nu\car\ab) a(R^c(\omega)) \\
&=& \sum_{c \in LAC(\omega)} \left((a_{\star}^T(P^c(\omega))\car\ab\right) a(R^c(\omega))\\
&=& \sum_{c \in LAC(\omega)} B^+\left(a_{\star}^T(P^c(\omega))\right) a(R^c(\omega)).
\end{eqnarray*}
The general formula is established by the following calculation, where $\tau$ is a tree:
\begin{eqnarray*}
a_{\star}^T(\omega) &=& \sum_{\nu,\tau} \left\langle \nu\tau, a_{\star}^T(\omega)\right\rangle \nu\tau \\
&=& \sum_{\nu,\tau} \left\langle (a_{\star} \nu)(a_{\star}\tau), \omega \right\rangle \nu\tau \\
&=& \sum_{(\omega) \in \Delta_c} \sum_{\nu, \tau} \left\langle a_{\star} \nu, \omega_{(1)} \right\rangle\left\langle a_{\star} \tau, \omega_{(2)} \right\rangle \nu\tau\\
&=& \sum_{(\omega) \in \Delta_c} (a_{\star}^T\omega_{(1)}) (\pi_{\OT}(a_{\star}^T \omega_{(2)})). 
\end{eqnarray*}
\end{proof}
\noindent As an example, the formula applied to the tree $\aababb$
yields $$a_{\star}^T(\aababb) = a(\aababb)B^+(\one) +
a(\ab)a(\aabb)B^+(\ab) = a(\aababb)\ab + a(\ab)a(\aabb)\aabb +
a(\ab)^3 \aababb.$$
See Table \ref{tab:subst} where this formula is computed for all forests up to order $4$, under the assumption that $a$ is an infinitesimal character.

\begin{proposition}
The map $a_{\star}^T$ is a character for the shuffle product:
\[a_{\star}^t(\omega_1 \shuffle \omega_2) = a_{\star}^t(\omega_1) \shuffle a_{\star}^t(\omega_2).\]
\end{proposition}
\begin{proof}
The shuffle product is dual to the deshuffle coproduct, so the result follows from Lemma \ref{coalg_subst} by dualization.
\end{proof}

\begin{remark}
There is a similar formula for the substitution of B-series, only with the coproduct $\Delta_N$ replaced by the Connes--Kreimer coproduct $\Delta_{CK} = \sum_{c \in AC(\tau)} P_{CK}^c(\tau) \otimes R_{CK}^c(\tau)$:
\begin{equation}
a_{\star}^T(\tau) = \sum_{c \in AC(\tau)}B^+\left(a_{\star}^T(P_{CK}^c(\tau))\right)a(R_{CK}^c(\tau))
\end{equation}
The proof of this formula is analogous to the proof of Theorem \ref{substlawformula}. This gives a recursive version of the coproduct in the substitution bialgebra $H_{CEFM}$ of \cite{calaque2009tih}
\end{remark}

\subsection{Backward error analysis and modified vector fields}\label{bea}
Recall the results on backward error analysis in \cite{chartier2005asl}: Given a B-series method $B_f(\alpha)$ there is a \emph{modified vector field} $\tilde{f}$ so that the B-series method applied to $f$ generates the exact flow of $\tilde{f}$. Moreover, $\tilde{f}$ can be written as a B-series with coefficients $\beta$ satisfying $\beta \star \gamma_{\text{Exact}} = \alpha$, where $\gamma_{\text{Exact}}$ is the coefficient function for the B-series of the exact flow, and $\star$ is the substitution law for characters in the Connes--Kreimer Hopf algebra $H_{CK}$. To generalize to LB-series, consider a numerical solution of the differential equation 
\begin{equation}
y' = f(y)\cdot y
\end{equation}
written in terms of a LB-series $\Bs_f(\alpha)$. We interpret it as the exact solution of a modified differential equation $y' = \tilde{f}(y) \cdot y$. As in the classical case, it turns out that the modified vector field can be written as a LB-series $\tilde{f} = \Bs_f(\beta)$. Furthermore, $\tilde{f}$ is such that 
\begin{equation}
\Bs_{\Bs_f(\beta)}(\gamma_{\text{Exact}}) = \Bs_f(\alpha),
\end{equation}
 where $\gamma_{\text{Exact}}$ represents the coefficients of the exact solution as described in Section \ref{flows}. This result follows by applying Proposition \ref{substBseries}.

\begin{theorem}
Let $\Bs_f(\alpha)$ be a LB-series method. There is a modified vector field $\tilde{f}$, given by $\tilde{f} = \Bs_f(\beta)$ such that $$\Bs_{\tilde{f}}(\gamma_{\text{Exact}}) = \Bs_f(\beta).$$ Moreover, $$\beta * \gamma_{\text{Exact}} = \alpha.$$
\end{theorem}

\begin{example}[The exponential Euler method]
The exponential Euler method is given by
$$y_{n+1} = \exp(hf(y_n))y_n.$$
In Example \ref{eulerMethod} the coefficients of the LB-series for this method was seen to be $\gamma = \ab$. To get the backward error, we calculate $\beta= Q(\ab) \circ e$, or $\log^*(Q(\ab))$ (cf. Section \ref{flows})
\begin{eqnarray*}
\beta = \ab -\frac{1}{2} \aabb + \frac{1}{3} \aaabbb + \frac{1}{12} \aababb - \frac{1}{12} \ab \aabb + \frac{1}{12} \aabb \ab - \frac{1}{4} \aaaabbbb - \frac{1}{12} \aaababbb - \frac{1}{12} \aaabbabb + \frac{1}{12} \ab \aaabbb - \frac{1}{12} \aaabbb \ab + \frac{1}{24} \ab \aababb - \frac{1}{24} \aababb\ab
\end{eqnarray*}
In the classical setting this logarithm has been studied as $\log^*(\delta)$, for a certain character $\delta$ \cite{chapoton2002rta, murua2006tha}.
\end{example}

\section{Implementation}\label{implement} As pointed out in Section \ref{sect:algebra}, the set of forests $\F$ can be generated recursively using a \emph{magmatic product} $\times$ defined on two forests $\omega_1$ and $\omega_2$ by 
\begin{equation}
\omega_1 \times \omega_2 = \omega_1B^+(\omega_2)
\end{equation}
by starting with the empty tree $\one$. Each forest in $\F$ can uniquely be written as a word in $\one$ and $\times$. Recall that if $\omega = \omega_1 \times \omega_2$, then we call $\omega_1$ the \emph{left part}, $\omega_L$, and $\omega_2$ the \emph{right part}, $\omega_R$, of $\omega$. All the basic algebraic operations used to construct the substitution law can be formulated in terms of this product:
\begin{itemize}
\item[] Concatenation: $\omega\,\one = \one\,\omega = \omega$, and $(\omega_1\times \omega_2)\,\omega_3 = \omega_1 \times (\omega_2 \,\omega_3)$.
\item[]Shuffle: $\omega \shuffle \one = \one \shuffle \omega = \omega$, and $\omega_1 \shuffle \omega_2 = (\omega_1 \shuffle \omega_{2L}) \times \omega_{1R} + (\omega_{1L} \shuffle \omega_2) \times \omega_{1R}$
\item[]Coproduct: $\Delta_N(\one) = \one \otimes \one$, and $\Delta_N(\omega) = \omega \otimes \one + \Delta_N(\omega_L) \shuffle \times \Delta_N(\omega_R)$
\end{itemize}
The formula (\ref{eq:magmasubstlaw}) for the substitution law in Theorem \ref{substlawformula} therefore lends itself well to implementation. 

\paragraph{Representing the free magma:} One way to represent the free magma is by using \emph{well-formed} words of parentheses `$($' and `$)$'. A word $w$ is well-formed if it is made of parentheses coming in pairs of one left and one right bracket, such that the left bracket appears on the left of the corresponding right bracket in $w$. For example,  $(())()$ is a well-formed word. The set of forests equipped with the product $\times$ is then isomorphic to this free magma via the recursion $\one = ()$,  $\omega_1 \times \omega_2 = (\omega_1)\omega_2.$

The authors have implemented a variant of the free magma, with elements represented by parentheses, and also the basic operations discussed in this paper. In future work, this implementation will be used to do backward error analysis on interesting test cases, like the dynamics of rigid bodies.

\section*{Acknowledgements}
We are grateful to Kurusch Ebrahimi-Fard, Dominique Manchon and Jon-Eivind Vatne for interesting and enlightening discussions, and to the anonymous referees for their valuable comments.  We would also like to acknowledge support from the Aurora Program, project 205042/V11.


\input{coprodTable}

\input{substlawTable}

\newpage
\bibliography{../../../bibliography/ref_alex}

\end{document}

%% file: coprodTable.tex
\begin{table}[ ]
  \centering
  \begin{equation*}
    \begin{array}{c@{\,\,}|@{\quad}l}
      \hline \\[-2mm]
      \omega & \Delta_N(\omega)  \\[1mm]
      \hline \\[-2mm]
      \one & \one\tpr\one \\[1mm]
      \ab & \ab\tpr\one+\one\tpr\ab \\[2mm]
      \aabb & \aabb\tpr\one+\ab\tpr\ab+\one\tpr\aabb  \\[2mm]
      \ab\ab & \ab\ab\tpr\one+\ab\tpr\ab+\one\tpr\ab\ab \\[2mm]
      \aaabbb &\aaabbb\tpr\one+\ab\tpr\aabb+\aabb\tpr\ab+\one\tpr\aaabbb \\[2.5mm]
      \aababb & \aababb\tpr\one+\ab\ab\tpr\ab+\ab\tpr\aabb+\one\tpr\aababb \\[2.5mm]
      \ab\aabb & \ab\aabb\tpr\one+2\ab\ab\tpr\ab      +\ab\tpr\aabb+\ab\tpr\ab\ab+\one\tpr\ab\aabb \\[2.5mm]
      \aabb\ab & \aabb\ab\tpr\one+\aabb\tpr\ab+\ab\tpr\ab\ab+\one\tpr\aabb\ab \\[2mm]
      \ab\ab\ab & \ab\ab\ab\tpr\one+\ab\ab\tpr\ab+\ab\tpr\ab\ab+\one\tpr\ab\ab\ab \\[2mm]
      \aaaabbbb & \aaaabbbb\tpr\one+\aaabbb\tpr\ab+\aabb\tpr\aabb+
      \ab\tpr\aaabbb+\one\tpr\aaaabbbb \\[2.5mm]
      \aaababbb & \aaababbb\tpr\one+\aababb\tpr\ab+\ab\ab\tpr\aabb+
      \ab\tpr\aaabbb+\one\tpr\aaababbb \\[2.5mm]
      \aabaabbb & \aabaabbb\tpr\one+\ab\aabb\tpr\ab+2\ab\ab\tpr
      \aabb+\ab\tpr\aaabbb+\ab\tpr\aababb+\one\tpr\aabaabbb \\[2.5mm]
      \aaabbabb & \aaabbabb\tpr\one+\aabb\ab\tpr\ab+\aabb\tpr
      \aabb+\ab\tpr\aababb+\one\tpr\aaabbabb \\[2.5mm]
      \aabababb & \aabababb\tpr\one+\ab\ab\ab\tpr\ab+\ab\ab\tpr      \aabb+\ab\tpr\aababb+\one\tpr\aabababb \\[2.5mm]
      \ab\aaabbb & \ab\aaabbb\tpr\one+\ab\aabb\tpr\ab+\aabb\ab\tpr\ab+\aabb\tpr      \ab\ab+2\ab\ab\tpr\aabb+\ab\tpr\ab\aabb+\ab\tpr\aaabbb+\one\tpr\ab\aaabbb \\[2.5mm]
      \aaabbb\ab & \aaabbb\ab\tpr\one+\aaabbb\tpr\ab+\aabb\tpr      \ab\ab+\ab\tpr\aabb\ab+\one\tpr\aaabbb\ab \\[2.5mm]
      \ab\aababb & \ab\aababb\tpr\one+3\ab\ab\ab\tpr\ab+      \ab\ab\tpr\ab\ab+2\ab\ab\tpr\aabb+     \ab\tpr\ab\aabb+\ab\tpr\aababb+\one\tpr\ab\aababb \\[2.5mm]
      \aababb\ab &  \aababb\ab\tpr\one+\aababb\tpr\ab+\ab\ab\tpr      \ab\ab+\ab\tpr\aabb\ab+\one\tpr\aababb\ab \\[2.5mm]
      \aabb\aabb & \aabb\aabb\tpr\one+\aabb\ab\tpr\ab+\ab\aabb\tpr\ab      +\aabb\tpr\aabb+2\ab\ab\tpr\ab\ab+\ab\tpr\ab\aabb+\ab\tpr\aabb\ab+ \one\tpr\aabb\aabb \\[2.5mm]
      \ab\ab\aabb & \ab\ab\aabb\tpr\one+3\ab\ab\ab\tpr\ab+      2\ab\ab\tpr\ab\ab+\ab\ab\tpr\aabb+      \ab\tpr\ab\ab\ab+\ab\tpr\ab\aabb+\one\tpr\ab\ab\aabb \\[2.5mm]
      \ab\aabb\ab & \ab\aabb\ab\tpr\one+\ab\aabb\tpr\ab+     2\ab\ab\tpr\ab\ab+\ab\tpr\ab\ab\ab+\ab\tpr\aabb\ab+\one\tpr\ab\aabb\ab \\[2.5mm]
     \aabb\ab\ab & \aabb\ab\ab\tpr\one+\aabb\ab\tpr\ab+    \aabb\tpr\ab\ab+\ab\tpr\ab\ab\ab+\one\tpr\aabb\ab\ab \\[2mm]
      \ab\ab\ab\ab & \ab\ab\ab\ab\tpr\one+\ab\ab\ab\tpr\ab+      \ab\ab\tpr\ab\ab+\ab\tpr\ab\ab\ab+\one\tpr\ab\ab\ab\ab \\[2mm] \hline
    \end{array}
  \end{equation*}
  \caption{Examples of the coproduct $\Delta_N$ \label{coprod}}
\end{table}

%% file: substlawTable.tex
\begin{table}[ ]
  \centering
  \begin{equation*}
    \begin{array}{c@{\,\,}|@{\quad}l}
      \hline \\
      \omega & \alpha_{\star}^T(\omega)  \\[2mm] 
      \hline \\[-1mm]
      \one & \one \\ 
      \ab & \alpha(\ab)\ab \\[2mm]     
      \aabb & \alpha(\aabb)\ab + \alpha(\ab)^2\aabb
      \\[1.5mm] 
      \ab\ab & \alpha(\ab)^2\ab\ab \\[2mm]
      \aaabbb & \alpha(\aaabbb)\ab +2 \alpha(\ab)\alpha(\aabb)\aabb + \alpha(\ab)^3 \aaabbb  \\[2.5mm] 
      \aababb & \alpha(\aababb)\ab + \alpha(\ab)\alpha(\aabb)\aabb + \alpha(\ab)^3\aababb\\[2.5mm]
      \ab\aabb & \alpha(\ab\aabb)\ab + \alpha(\ab)\alpha(\aabb) \ab\ab + \alpha(\ab)^3\ab\aabb \\[2.5mm]
      \aabb\ab &  \alpha(\aabb\ab)\ab + \alpha(\ab)\alpha(\aabb)\ab\ab + \alpha(\ab)^3 \aabb\ab\\[2.5mm]
      \ab\ab\ab &  \alpha(\ab)^3 \ab\ab\ab\\[2mm]
      \aaaabbbb &  \alpha(\aaaabbbb)\ab + 2\alpha(\ab)\alpha(\aaabbb)\aabb + \alpha(\aabb)^2\aabb + 3\alpha(\ab)^2 \alpha(\aabb)\aaabbb + \alpha(\ab)^4 \aaaabbbb\\[2.5mm]
      \aaababbb &  \alpha(\aaababbb)\ab + \alpha(\ab)\alpha(\aaabbb)\aabb + \alpha(\ab)\alpha(\aababb) \aabb + \alpha(\ab)^2 \alpha(\aabb)\left(\aababb + \aaabbb\right) + \alpha(\ab)^4 \aaababbb\\[2.5mm]  
      \aabaabbb &  \alpha(\aabaabbb)\ab + \alpha(\ab)\alpha(\ab\aabb) \aabb + \alpha(\ab)\alpha(\aaabbb)\aabb + \alpha(\ab)\alpha(\aababb)\aabb + 3\alpha(\ab)^2\alpha(\aabb)\aababb + \alpha(\ab)^4 \aabaabbb\\[2.5mm]
      \aaabbabb &  \alpha(\aaabbabb)\ab + \alpha(\ab)\alpha(\aabb\ab)\aabb + \alpha(\ab)\alpha(\aababb) \aabb + \alpha(\aabb)^2\aabb + \alpha(\ab)^2\alpha(\aabb)\left(\aababb + \aaabbb\right) + \alpha(\ab)^4 \aaabbabb\\[2.5mm] 
      \aabababb &  \alpha(\aabababb)\ab + \alpha(\ab)\alpha(\aababb)\aabb + \alpha(\ab)^2 \alpha(\aabb)\aababb + \alpha(\ab)^4\aabababb\\[2.5mm]  
      \ab\aaabbb &  \alpha(\ab\aaabbb)\ab + \alpha(\ab)\alpha(\ab\aabb)\aabb + \alpha(\ab)\alpha(\aaabbb)\ab\ab + 2\alpha(\ab)^2 \alpha(\aabb)\ab\aabb + \alpha(\ab)^4\ab\aaabbb\\[2.5mm]  
      \aaabbb\ab &  \alpha(\aaabbb\ab) \ab + \alpha(\ab)\alpha(\aabb\ab)\aabb + \alpha(\ab)\alpha(\aaabbb)\ab\ab + 2\alpha(\ab)^2\alpha(\aabb)\aabb\ab + \alpha(\ab)^4 \aaabbb\ab\\[2.5mm]   
      \ab\aababb &  \alpha(\ab\aababb)\ab + \alpha(\ab)\alpha(\ab\aabb)\aabb + \alpha(\ab)\alpha(\aababb)\ab\ab + \alpha(\ab)^2 \alpha(\aabb)\ab\aabb + \alpha(\ab)^4 \ab\aababb\\[2.5mm]   
      \aababb\ab &  \alpha(\aababb\ab)\ab + \alpha(\ab)\alpha(\aabb\ab)\aabb + \alpha(\ab)\alpha(\aababb)\ab\ab + \alpha(\ab)^2\alpha(\aabb)\aabb\ab + \alpha(\ab)^4\aababb\ab\\[2.5mm]   
      \aabb\aabb & \alpha(\aabb)^2\ab\ab + \alpha(\ab)^2\alpha(\aabb)\left(\ab\aabb + \aabb\ab\right) + \alpha(\ab)^4\aabb\aabb\\[2.5mm]
      \ab\ab\aabb &  \alpha(\ab\ab\aabb)\ab + \alpha(\ab)\alpha(\ab\aabb)\ab\ab + \alpha(\ab)^2\alpha(\aabb)\ab\ab\ab + \alpha(\ab)^4 \ab\ab\aabb\\[2.5mm] 
      \ab\aabb\ab &  \alpha(\ab\aabb\ab)\ab + \alpha(\ab)^2\alpha(\aabb)\ab\ab\ab + \alpha(\ab)^4 \ab\aabb\ab\\[2.5mm] 
      \aabb\ab\ab &  \alpha(\aabb\ab\ab)\ab + \alpha(\ab)\alpha(\aabb\ab)\ab\ab + \alpha(\ab)^2\alpha(\aabb) \ab\ab\ab + \alpha(\ab)^4\aabb\ab\ab\\[2.5mm] 
      \ab\ab\ab\ab & \alpha(\ab)^4\ab\ab\ab\ab\\[2mm] \hline 
    \end{array}
  \end{equation*}
  \caption{Examples of the substitution character $\alpha_{\star}^T$, where $\alpha$ is an infinitesimal character,  for all forests up to and including order four. \label{tab:subst}}
\end{table}